\pdfoutput=1
\documentclass[reqno, 12pt]{amsart}

\def\ssign{\textsection\nobreak\hspace{1pt plus 0.3pt}}
\makeatletter
\newif\if@sectsign \@sectsigntrue
\def\@seccntformat#1{%
	\protect\textup{\protect\@secnumfont
		\ifnum\pdfstrcmp{#1}{section}=\z@
			\if@sectsign \protect\ssign\csname the#1\endcsname\protect\enspace
			\else \csname the#1\endcsname\protect\@secnumpunct\fi
		\else
			\csname the#1\endcsname\protect\@secnumpunct
		\fi
	}%
}
\g@addto@macro\appendix{\@sectsignfalse}
\makeatother

\usepackage{amsmath,amssymb,amsthm}
\usepackage{mathrsfs}
\usepackage{mathabx}\changenotsign
\usepackage{dsfont}
\usepackage{scalerel}
\usepackage{graphicx}
\usepackage{trimclip}

\usepackage[T1]{fontenc}
\usepackage{lmodern}
\usepackage[british]{babel}
\usepackage[babel]{microtype}

\usepackage{geometry}
\usepackage{xcolor}
\usepackage[backref]{hyperref}
\hypersetup{
	colorlinks,
	linkcolor={red!60!black},
	citecolor={green!60!black},
	urlcolor={blue!60!black}
}

\usepackage[open,openlevel=2,atend]{bookmark}

\usepackage[abbrev,msc-links,backrefs]{amsrefs}
\usepackage{doi}

\renewcommand{\PrintDOI}[1]{\doi{#1}}

\numberwithin{equation}{section}
\numberwithin{figure}{section}

\usepackage{calc}
\usepackage{enumitem}

\theoremstyle{plain}
\newtheorem{thm}{Theorem}[section]

\newtheorem{prop}[thm]{Proposition}

\theoremstyle{definition}
\newtheorem{dfn}[thm]{Definition}

\theoremstyle{remark}

\let\eps=\varepsilon
\let\theta=\vartheta
\let\rho=\varrho
\let\phi=\varphi

\let\polishlcross=\l
\DeclareRobustCommand{\l}{\ifmmode\ell\else\polishlcross\fi}
\ifdefined{\def\l{l}}\fi

\def\EE{\mathds E}

\def\NN{\mathds N}
\def\PP{\mathds P}

\makeatletter
\def\@defcal#1{\expandafter\def\csname c#1\endcsname{{\mathcal{#1}}}}
\def\@defscr#1{\expandafter\def\csname cc#1\endcsname{{\mathscr{#1}}}}
\def\@deffrak#1{\expandafter\def\csname f#1\endcsname{{\mathfrak{#1}}}}
\@tfor\next:=ABCDEFGHIJKLMNOPQRSTUVWXYZ\do{\expandafter\@defcal\next}
\@tfor\next:=ABCDEFGHIJKLMNOPQRSTUVWXYZ\do{\expandafter\@defscr\next}
\@tfor\next:=ABCDEFGHIJKLMNOPQRSTUVWXYZ\do{\expandafter\@deffrak\next}
\@tfor\next:=abcdefghjklmnopqrstuvwxyz\do{\expandafter\@deffrak\next}
\makeatother

\makeatletter
\def\moverlay{\mathpalette\mov@rlay}
\def\mov@rlay#1#2{\leavevmode\vtop{   \baselineskip\z@skip
		\lineskiplimit-\maxdimen
\ialign{\hfil$\m@th#1##$\hfil\cr#2\crcr}}}
\newcommand{\charfusion}[3][\mathord]{
	#1{\ifx#1\mathop\vphantom{#2}\fi
		\mathpalette\mov@rlay{#2\cr#3}
	}
\ifx#1\mathop\expandafter\displaylimits\fi}
\makeatother

\newcommand{\dcup}{\charfusion[\mathbin]{\cup}{\cdot}}

\DeclareFontFamily{U}  {MnSymbolC}{}
\DeclareSymbolFont{MnSyC}         {U}  {MnSymbolC}{m}{n}
\DeclareFontShape{U}{MnSymbolC}{m}{n}{
	<-6>  MnSymbolC5
	<6-7>  MnSymbolC6
	<7-8>  MnSymbolC7
	<8-9>  MnSymbolC8
	<9-10> MnSymbolC9
	<10-12> MnSymbolC10
<12->   MnSymbolC12}{}
\DeclareMathSymbol{\powerset}{\mathord}{MnSyC}{180}

\def\tand{\ \text{and}\ }
\def\qand{\quad\text{and}\quad}
\def\qqand{\qquad\text{and}\qquad}

\let\setminus=\smallsetminus
\let\emptyset=\varnothing

\let\to=\lra

\makeatletter
\def\@shortto@sub_#1{_{\begingroup\let\to\rightarrow#1\endgroup}}
\newcommand{\@withshortto}[1]{#1\@ifnextchar_\@shortto@sub\relax}
\@for\@op:=lim,limsup,liminf,varlimsup,varliminf,varinjlim,varprojlim,projlim,injlim\do{%
	\@ifundefined{\@op}{}{%
		\expandafter\let\csname @save@\@op\expandafter\endcsname\csname\@op\endcsname
		\expandafter\edef\csname\@op\endcsname{\noexpand\@withshortto\expandafter\noexpand\csname @save@\@op\endcsname}%
	}%
}
\makeatother

\makeatletter
\@ifpackageloaded{mathabx}{}{%
	\DeclareFontFamily{U}{matha}{\hyphenchar\font45}%
	\DeclareFontShape{U}{matha}{m}{n}{<-11> matha10 <11-> matha12}{}%
	\DeclareSymbolFont{matha}{U}{matha}{m}{n}%
}%
\DeclareMathSymbol{\hp@RU}{\mathord}{matha}{"E1}
\DeclareMathSymbol{\hp@LU}{\mathord}{matha}{"E0}
\newcommand{\hpsc}{0.6}\newcommand{\hpgapn}{3}\newcommand{\hpbot}{0.68}
\newcommand{\hpbarb}{0.55}\newcommand{\hpoh}{0.1}\newcommand{\hpminf}{0.5}\newcommand{\hpov}{0.04}\newcommand{\hpdx}{0}
\newcommand{\rhpinL}{0.1}\newcommand{\rhpinR}{0.0}
\newcommand{\lhpinL}{0.1}\newcommand{\lhpinR}{0.0}
\newsavebox\hpA\newsavebox\hpG\newsavebox\hpB\newsavebox\hpF\newsavebox\hpL\newsavebox\hpSl
\newdimen\hp@li\newdimen\hp@L\newdimen\hp@gap\newdimen\hp@dx\newdimen\hp@bs\newdimen\hp@oh\newdimen\hp@rw\newdimen\hp@ov\newdimen\hp@ri\newdimen\hp@xi\newif\ifhp@cramp
\def\hp@nil{}
\def\@hpfirst#1#2\hp@nil{\ifcat A\noexpand#1\sbox\hpF{$\m@th\SavedStyle#1$}\else\sbox\hpF{$\m@th\SavedStyle#1#2$}\fi}
\def\@hplast#1{\ifx#1\hp@nil\else\def\@hplastv{#1}\expandafter\@hplast\fi}
\newcommand{\hp@setA}[1]{%
	\sbox\hpA{$\m@th\SavedStyle\kern-\nulldelimiterspace\radical\z@{#1}$}%
	\dimen@=1.25\fontdimen8\textfont3
	\if D\m@switch \dimen@=\fontdimen8\textfont3 \advance\dimen@.25\fontdimen5\textfont2 \fi
	\if S\m@switch \dimen@=1.25\fontdimen8\scriptfont3 \fi
	\if s\m@switch \dimen@=1.25\fontdimen8\scriptscriptfont3 \fi
	\advance\dimen@-\ht\hpA \ht\hpA=-\dimen@}
\newcommand{\hp@common}[3]{%
	\ifhp@cramp \hp@setA{#3}\else \sbox\hpA{$\m@th\SavedStyle#3$}\fi
	\@hpfirst#3\hp@nil \@hplast#3\hp@nil \sbox\hpL{$\m@th\SavedStyle\@hplastv$}%
	\hp@li=#1\wd\hpF \hp@ri=#2\wd\hpL
	\hp@L=\wd\hpA \advance\hp@L-\hp@li \advance\hp@L-\hp@ri
	\hp@bs=\hpbarb\wd\hpG \hp@oh=\hpoh\ht\hpG
	\hp@xi=\fontdimen8\textfont3
	\if S\m@switch\hp@xi=\fontdimen8\scriptfont3\fi \if s\m@switch\hp@xi=\fontdimen8\scriptscriptfont3\fi
	\hp@dx=\hpdx\wd\hpA \hp@ov=\hpov\wd\hpG}
\newcommand{\hp@overgap}{\hp@gap=\hpgapn\hp@xi \@tempdima=\hpbot\ht\hpG \advance\hp@gap-\@tempdima}
\newcommand{\hp@overbox}[2]{\mathord{\vbox{\offinterlineskip\ialign{##\cr
	\hbox to\wd\hpA{\kern#1#2\hfil}\cr
	\noalign{\kern\hp@gap}\hbox{\usebox\hpA}\cr}}}}
\newcommand{\hpdrawR}{%
	\ifdim\hp@L<\wd\hpG
		\clipbox{\dimexpr\wd\hpG-\hp@L\relax{} 0pt 0pt \dimexpr-\hp@oh\relax{}}{\usebox\hpG}%
	\else
		\hbox to\dimexpr\hp@rw+\hp@ov\relax{\cleaders\hbox{\usebox\hpSl}\hfil}\kern-\hp@ov\usebox\hpB%
	\fi}
\newcommand{\hp@bR}[3]{\hp@cramptrue
	\sbox\hpG{\scalebox{\hpsc}{$\m@th\SavedStyle\hp@RU$}}%
	\hp@common{#1}{#2}{#3}%
	\ifdim\hp@L<\hpminf\wd\hpG \hp@L=\hpminf\wd\hpG \fi
	\hp@rw=\hp@L \advance\hp@rw-\wd\hpG \advance\hp@rw\hp@bs
	\sbox\hpB{\clipbox{\dimexpr\hp@bs\relax{} 0pt 0pt \dimexpr-\hp@oh\relax{}}{\usebox\hpG}}%
	\sbox\hpSl{\clipbox{\dimexpr\wd\hpG*1/5\relax{} 0pt \dimexpr\wd\hpG*39/50\relax{} 0pt}{\usebox\hpG}}%
	\hp@overgap \hp@overbox{\dimexpr\wd\hpA-\hp@ri-\hp@L+\hp@dx\relax}{\hpdrawR}}
\newcommand{\hpdrawL}{%
	\ifdim\hp@L<\wd\hpG
		\clipbox{0pt 0pt \dimexpr\wd\hpG-\hp@L\relax{} \dimexpr-\hp@oh\relax{}}{\usebox\hpG}%
	\else
		\usebox\hpB\kern-\hp@ov\hbox to\dimexpr\hp@rw+\hp@ov\relax{\cleaders\hbox{\usebox\hpSl}\hfil}%
	\fi}
\newcommand{\hp@bL}[3]{\hp@cramptrue
	\sbox\hpG{\scalebox{\hpsc}{$\m@th\SavedStyle\hp@LU$}}%
	\hp@common{#1}{#2}{#3}%
	\ifdim\hp@L<\hpminf\wd\hpG \hp@L=\hpminf\wd\hpG \fi
	\hp@rw=\hp@L \advance\hp@rw-\wd\hpG \advance\hp@rw\hp@bs
	\sbox\hpB{\clipbox{0pt 0pt \dimexpr\hp@bs\relax{} \dimexpr-\hp@oh\relax{}}{\usebox\hpG}}%
	\sbox\hpSl{\clipbox{\dimexpr\wd\hpG*39/50\relax{} 0pt \dimexpr\wd\hpG*1/5\relax{} 0pt}{\usebox\hpG}}%
	\hp@overgap \hp@overbox{\dimexpr\hp@li+\hp@dx\relax}{\hpdrawL}}
\NewDocumentCommand\rharp{O{\rhpinL}O{\rhpinR}m}{\ThisStyle{\hp@bR{#1}{#2}{#3}}}
\NewDocumentCommand\lharp{O{\lhpinL}O{\lhpinR}m}{\ThisStyle{\hp@bL{#1}{#2}{#3}}}
\makeatother

\def\bdelta{\boldsymbol{\delta}}

\def\js{j_{\star}}
\def\Js{J_{\star}}

\def\UJs{U_{\Js}}
\def\Hs{H_{\star}}
\def\Gs{G_{\star}}

\def\phis{\phi_{G}}

\makeatletter
\newcommand{\oset}[3][0ex]{\mathrel{\mathop{#3}\limits^{
			\vbox to#1{\kern-2.1\ex@
				\hbox{$\scriptstyle#2$}\vss}}}}
\makeatother
\DeclareMathSymbol{*}{\mathbin}{symbols}{"03}

\DeclareSymbolFont{stmry}{U}{stmry}{m}{n}
\DeclareMathSymbol\arrownot\mathrel{stmry}{"58}
\DeclareMathSymbol\Arrownot\mathrel{stmry}{"59}

\makeatletter
\newcommand{\pushright}[1]{\ifmeasuring@#1\else\omit\hfill$\displaystyle#1$\fi\ignorespaces}
\newcommand{\pushleft}[1]{\ifmeasuring@#1\else\omit$\displaystyle#1$\hfill\fi\ignorespaces}
\makeatother

\DeclareMathOperator{\Var}{Var}
\DeclareMathOperator{\Cov}{Cov}
\DeclareMathOperator{\ER}{ER}
\begin{document}
\title[Canonical Ramsey numbers for partite hypergraphs]{Canonical Ramsey numbers for partite hypergraphs}

\author[M.~Az\'ocar]{Mat\'ias Az\'ocar Carvajal}
\address{Fachbereich Mathematik, Universit\"at Hamburg, Hamburg, Germany}
\email{matias.azocar.carvajal@uni-hamburg.de}
\author[G.~Santos]{Giovanne Santos}
\address{Departamento de Ingeniería Matemática, Universidad de Chile, Santiago, Chile}
\email{gsantos@dim.uchile.cl}
\author[M.~Schacht]{Mathias Schacht}
\address{Fachbereich Mathematik, Universit\"at Hamburg, Hamburg, Germany}
\email{schacht@math.uni-hamburg.de}

\thanks{The first two authors have been supported by ANID and DAAD under
	ANID-PFCHA/Doctorado Acuerdo Bilateral DAAD Becas Chile/2023-62230021
	and by ANID Becas/Doctorado Nacional~21221049.}


\begin{abstract}
	We show that canonical Ramsey numbers for partite hypergraphs grow single-exponentially
	for any fixed uniformity.
\end{abstract}

\maketitle

\section{Introduction}
\label{sec:introduction}
Erd\H os and Rado~\cite{ER50} established the canonical Ramsey theorem, which generalises Ramsey's theorem~\cite{R30} to an unbounded number of colours.
In that work Erd\H os and Rado characterised all canonical colour patterns that are unavoidable in colourings of edge sets of sufficiently large hypergraphs.
The canonical Ramsey number $\ER(K_t^{(k)})$, also referred to as the Erd\H os--Rado number of $K_t^{(k)}$, is the smallest integer~$n$ such that
any edge colouring $\phi\colon E(K_n^{(k)})\to\NN$ of the complete~$k$-uniform hypergraph on~$n$ vertices
yields a canonical copy of $K^{(k)}_t$, i.e., a copy which exhibits one of those unavoidable colour patterns.
(The precise definition of these colour patterns is not important at this point. However, we remark that
for their definition the underlying vertex set is assumed to be ordered.)
We consider quantitative aspects of this theorem. For the Erd\H os--Rado theorem
discussed above, it follows from the work of Erd\H os, Hajnal, and Rado~\cite{EHR65}*{\ssign16.4} (see also reference~\cite{EH89}*{(4.2)}),
the work of Lefmann and R\"odl~\cite{LR95}, and the work of Shelah~\cite{S96} that the lower and the upper bound
on~$\ER(K_t^{(k)})$ grow as~$(k-1)$-fold iterated exponentials in polynomials of~$t$
(see also reference~\cite{RRSSS22}*{\ssign4}). In other words,
in terms of the number of exponentiations the canonical Ramsey number and the non-canonical Ramsey number (for at least $4$ colours) display the same behaviour.

We study Erd\H os--Rado numbers for~$k$-partite~$k$-uniform hypergraphs. The extremal problem for~$k$-partite~$k$-uniform hypergraphs is degenerate
and as a result the Ramsey number grows much slower. In fact, owing to the work of K\H{o}v\'{a}ri, S\'os, and Tur\'an~\cite{KST54} and of
Erd\H os~\cite{E64} those Ramsey numbers for any fixed number of colours grow only exponentially, and random colourings yield
a matching lower bound. Roughly speaking, we show that canonical Ramsey numbers for partite hypergraphs exhibit the same behaviour and, in fact,
these extremal results will be crucial in the proof.
We recall the definition of canonical colourings for partite hypergraphs.

\begin{dfn} For a~$k$-partite~$k$-uniform hypergraph $H=(V_1\dcup\dots\dcup V_k,E)$, a set~$J\subseteq [k]$, and
	an edge $e\in E$ we write
	\[
		e_J = e \cap\bigcup_{j\in J} V_j
	\]
	for the restriction of~$e$ to the vertex classes indexed by~$J$.
	We say a colouring $\phi\colon E\to\NN$
	is \emph{$J$-canonical} if for all edges~$e$, $e'\in E$ we have
	\[
		\phi(e)=\phi(e')
		\qquad
		\Longleftrightarrow
		\qquad
		e_J=e'_J\,.
	\]
	Moreover, we say the colouring is \emph{canonical} if it is~$J$-canonical for some $J\subseteq [k]$
	and a subhypergraph is a \emph{canonical copy} if the restriction of $\phi$ to its edges
	is~$J$-canonical for some $J\subseteq [k]$.
\end{dfn}

Since $e_\emptyset=\emptyset$ for all edges $e\in E$, we observe that $\emptyset$-canonical colourings are monochromatic. At the other extreme, the $[k]$-canonical colourings are injective
and as usual we refer to those as \emph{rainbow} colourings.

Similarly to the above, we define $\ER(K^{(k)}_{t,\dots,t})$ as the smallest integer~$n$ such that every colouring $\phi\colon E(K^{(k)}_{n,\dots,n})\to \NN$ of the edges of the complete~$k$-partite~$k$-uniform hypergraph with vertex classes of size~$n$
yields a canonical copy of $K^{(k)}_{t,\dots,t}$. It follows from the work of Rado~\cite{R54} that these numbers exist and a simple probabilistic argument
employing a random colouring with~$t^{k}-1$ colours shows
\begin{equation}\label{eq:ER-lb}
	\ER(K^{(k)}_{t,\dots,t})
	\geq
	t^{(1-o(1))t^{k-1}}\,,
\end{equation}
where $o(1)\to 0$ as $t\to\infty$. We establish a comparable upper bound, which resolves a problem raised by
Dob\'ak and Mulrenin~\cite{DM}.

\begin{thm}\label{thm:main}
	For sufficiently large~$t$ we have $\ER(K^{(2)}_{t,t})\leq t^{4t}$ and $\ER(K^{(3)}_{t,t,t})\leq t^{30t^3}$. Moreover,
	for every $k\geq 4$ and~$t$ sufficiently large we have
	\[
		\ER(K^{(k)}_{t,\dots,t})\leq t^{t^{k^2}}\,.
	\]
\end{thm}
In view of the lower bound~\eqref{eq:ER-lb},
Theorem~\ref{thm:main} for $k=2$ is optimal up to the factor~$4$ in the exponent.
In fact, our proof shows that $4$ can be replaced by $3+\eps$ for every $\eps>0$ and
sufficiently large~$t$. In earlier work, Kostochka, Mubayi, and Verstra\"ete~\cite{KMV17}
obtained a bound of the form $t^{\textrm{poly}(t)}$. More recently,
Dob\'ak and Mulrenin~\cite{DM} established $\ER(K^{(2)}_{t,t})\leq t^{(8+\eps)t}$ (see also the work
of Gishboliner, Milojevi\'c, Sudakov, and Wigderson~\cite{GMSW} for a related result concerning a
different variant of Erd\H os--Rado numbers).

For $k=3$ there is a more substantial gap between the lower bound~\eqref{eq:ER-lb}
and the upper bound provided by Theorem~\ref{thm:main}. In view of that, it would be interesting to decide whether the cubic exponent $30t^3$
in the upper bound could be improved to be quadratic.

For larger values of~$k$ the gap between the lower and the upper bound widens and in
the proof of Theorem~\ref{thm:main} we made no attempt to obtain the optimal constant in front of the exponent~$k^2$. However, our method seems
to fall short of obtaining a factor smaller than~$1/2$. In particular, we leave it open whether the $k^2$ can be improved to $o(k^2)$
or even to~$O(k)$ for~$k\to\infty$, as suggested by the lower bound~\eqref{eq:ER-lb}.

\section{Preparations}
\label{sec:prep}
The proof of Theorem~\ref{thm:main} follows the approach of Dob\'ak and Mulrenin~\cite{DM}.
It is based on an unbalanced variant of Erd\H os' extremal result for partite hypergraphs, which we introduce in~\ssign\ref{sec:E64}.
Another key idea, often used to locate rainbow copies, is to consider bounded colourings, and we introduce those in~\ssign\ref{sec:bdd-cols}.
\subsection{Extremal problem for partite hypergraphs}
\label{sec:E64}
The proof of Theorem~\ref{thm:main} relies on the extension of the K\H{o}v\'{a}ri--S\'os--Tur\'an theorem~\cite{KST54} to hypergraphs due to Erd\H os~\cite{E64}.
For completeness we include the proof of the following variant of that result for partite hypergraphs with vertex classes of different sizes.
For the inductive proof it is helpful to consider $1$-uniform hypergraphs.
\begin{prop}\label{prop:unbE64}
	For $k\geq 1$ let $H=(V_1\dcup \dots\dcup V_k,E)$ be a~$k$-partite~$k$-uniform hypergraph of density $d=\frac{|E|}{|V_1|\cdots|V_k|}$.
	If
	for some positive integers $t_1,\dots,t_k$ we have
	\begin{equation}
		\label{eq:aE64}
		\Big(\frac{d}{4^{k-1}}\Big)^{\prod_{j<i}t_j}|V_i| \geq  2t_i
	\end{equation}
	for every $i\in[k]$,
	then the set $K^{(k)}_{t_1,\dots,t_k}(H)$ of complete~$k$-partite~$k$-uniform hypergraphs in~$H$
	with vertex classes $U_j\subseteq V_j$ and $|U_j|=t_j$ for every $j\in[k]$ satisfies
	\begin{equation}\label{eq:cE64}
		\big|K^{(k)}_{t_1,\dots,t_k}(H)\big|>\Big(\frac{d}{2^{2k-1}}\Big)^{\prod_{j\in[k]}t_j}\prod_{j\in[k]}\binom{|V_j|}{t_j}\,.
	\end{equation}
\end{prop}
In the proof of Proposition~\ref{prop:unbE64} below, we shall use the following lower bound on the binomial coefficient
\begin{equation}\label{eq:binom}
	\binom{dn}{t}>\Big(\frac{d}{2}\Big)^t\binom{n}{t}\,,
\end{equation}
which holds under the assumption $dn\geq 2t$.
\begin{proof}
	The proof is by induction on~$k$. For $k=1$ assumption~\eqref{eq:aE64} reduces to $d|V_1|\geq 2t_1$
	and inequality~\eqref{eq:binom} yields
	$\binom{d|V_1|}{t_1}> (d/2)^{t_1}\binom{|V_1|}{t_1}$, which establishes conclusion~\eqref{eq:cE64}.

	For the inductive step consider a~$k$-partite~$k$-uniform hypergraph $H=(V_1\dcup \dots\dcup V_k,E)$ of
	density~$d$. For every vertex $v\in V_k$ we denote by $H(v)$ the $(k-1)$-uniform hypergraph
	on vertex classes $V_1,\dots,V_{k-1}$ defined by the link of~$v$
	and we set
	\[
		K(v)=K^{(k-1)}_{t_1,\dots,t_{k-1}}(H(v))\,,
	\]
	i.e.,~$K(v)$ is
	the set of complete $(k-1)$-partite $(k-1)$-uniform hypergraphs in $H(v)$ with vertex classes $U_j\subseteq V_j$
	and~$|U_j|=t_j$ for every $j\in[k-1]$. Let $V_k^\star\subseteq V_k$ be the set of
	those vertices~$v$ that satisfy
	\[
		e(H(v))
		\geq
		\frac{|E|}{2|V_k|}\,.
	\]
	In particular,
	\begin{equation}\label{eq:ind_assumption}
		\sum_{v\in V_k^{\star}}e(H(v))\geq \frac{|E|}{2}
	\end{equation}
	and, more importantly, for every $v\in V_k^\star$ the link hypergraph $H(v)$ satisfies the assumption~\eqref{eq:aE64} with~$k$ replaced
	by~$k-1$. Consequently, the inductive hypothesis applied to~$H(v)$ for every $v\in V_k^\star$ yields
	\[
		\sum_{v\in V_k}|K(v)|
		\geq
		\sum_{v\in V_k^{\star}}|K(v)|
		>
		\sum_{v\in V_k^{\star}}\bigg(\frac{e(H(v))}{2^{2k-3}|V_1|\cdots|V_{k-1}|}\bigg)^{\prod_{j\in[k-1]}t_j}\prod_{j\in[k-1]}\binom{|V_j|}{t_j}\,.
	\]
	Jensen's inequality with weights $1/|V_k^{\star}|$ for every $v\in V_k^{\star}$
	combined with~\eqref{eq:ind_assumption} tells us
	\begin{align}\label{eq:E64-1}
		\sum_{v\in V_k}|K(v)|\geq
		\sum_{v\in V_k^{\star}}|K(v)|
		 & >
		|V_k^{\star}|\cdot\bigg(\frac{|E|/2}{2^{2k-3}|V_1|\cdots|V_{k-1}|\cdot |V_k^{\star}|}\bigg)^{\prod_{j\in[k-1]}t_j}\prod_{j\in[k-1]}\binom{|V_j|}{t_j}
		\nonumber \\
		 & \geq
		|V_k|\cdot\bigg(\frac{|E|/2}{2^{2k-3}|V_1|\cdots|V_{k-1}|\cdot |V_k|}\bigg)^{\prod_{j\in[k-1]}t_j}\prod_{j\in[k-1]}\binom{|V_j|}{t_j}
		\nonumber \\
		 & =
		|V_k|
		\cdot
		\Big(\frac{d}{4^{k-1}}\Big)^{\prod_{j\in[k-1]}t_j}\prod_{j\in[k-1]}\binom{|V_j|}{t_j}\,.
	\end{align}
	On the other hand, for a copy~$K$ of $K^{(k-1)}_{t_1,\dots,t_{k-1}}$ in the complete
	$(k-1)$-partite $(k-1)$-uniform hypergraph with vertex set $V_1\dcup \dots \dcup V_{k-1}$
	we consider the set
	\[
		V_k(K)=\{v\in V_k\colon K\in K(v)\}\,.
	\]
	Clearly, summing over all such copies~$K$ yields
	\[
		\sum_{K}|V_k(K)|=\sum_{v\in V_k}|K(v)|
	\]
	and another application of Jensen's inequality implies
	\[
		\big|K^{(k)}_{t_1,\dots,t_k}(H)\big|
		=
		\sum_{K}\binom{|V_k(K)|}{t_k}
		\geq
		\prod_{j\in[k-1]}\binom{|V_j|}{t_j}
		\cdot \binom{\sum_{v\in V_k}|K(v)| / \prod_{j\in[k-1]}\binom{|V_j|}{t_j}}{t_k}\,.
	\]
	Finally, using the assumption~\eqref{eq:aE64} in the second inequality below, we arrive at
	\begin{align*}
		\big|K^{(k)}_{t_1,\dots,t_k}(H)\big|
		 & \overset{\eqref{eq:E64-1}}{>}
		\prod_{j\in[k-1]}\binom{|V_j|}{t_j} \cdot \binom{|V_k|\cdot\big(\frac{d}{4^{k-1}}\big)^{\prod_{j\in[k-1]}t_j}}{t_k} \\
		 & \overset{\eqref{eq:binom}}{>}
		\prod_{j\in[k-1]}\binom{|V_j|}{t_j} \cdot
		\bigg(\frac{1}{2}
		\Big(\frac{d}{4^{k-1}}\Big)^{\prod_{j\in[k-1]}t_j}\bigg)^{t_k}\binom{|V_k|}{t_k}                                    \\
		 & \overset{\phantom{\eqref{eq:E64-1}}}{\geq}
		\Big(\frac{d}{2^{2k-1}}\Big)^{\prod_{j\in[k]}t_j}\prod_{j\in[k]}\binom{|V_j|}{t_j}\,,
	\end{align*}
	and this concludes the proof of Proposition~\ref{prop:unbE64}.
\end{proof}

\subsection{Bounded colourings}
\label{sec:bdd-cols}
For a~$k$-partite~$k$-uniform hypergraph $H=(V_1\dcup\dots\dcup V_k,E)$ and~$J\subseteq [k]$ we
write~$V_J$ for the set of $|J|$-element vertex sets intersecting the vertex classes indexed by~$J$,
i.e., for $J=\{i_1,\dots,i_{|J|}\}$,
\[
	V_J = \big\{ \{v_1,\dots,v_{|J|}\} \colon v_j\in V_{i_j}\ \text{for all}\ j\in\{1,\dots,|J|\}\big\}\,.
\]
Clearly, we have $|V_\emptyset|=1$ and $|V_J|=\prod_{j\in J}|V_j|$. Moreover, the restriction~$e_J$ of any edge $e\in E$ is an element of~$V_J$.

Roughly speaking, the proof of Theorem~\ref{thm:main} pivots on a classification of the colourings given by the following definition.
\begin{dfn}
	We say a colouring $\phi\colon E\to \NN$ of the edge set of a
	$k$-partite~$k$-uniform hypergraph $H=(V_1\dcup\dots\dcup V_k,E)$
	is \emph{$(\delta,J)$-bounded} for some $\delta>0$ and some set $J\subsetneq [k]$
	if all but
	at most $\delta |V_J|$ of the $|J|$-tuples $S\in V_J$ satisfy
	\[
		\big|\{e\in E\colon e_J=S \tand \phi(e)=\l\}\big|
		\leq
		\delta \big|V_{[k]\setminus J}\big|
	\]
	for every colour $\l\in\NN$.
	For $j=0,\dots,k-1$ we say the colouring $\phi$ is \emph{$(\delta,j)$-bounded} if it is
	\emph{$(\delta,J)$-bounded} for every~$j$-element set $J\in [k]^{(j)}$.

	Moreover, we say $\phi$ is \emph{$\bdelta$-bounded} for $\bdelta=(\delta_0,\dots,\delta_{k-1})\in(0,1]^k$ if
	$\phi$ is $(\delta_j,j)$-bounded for every $j=0,\dots,k-1$.
\end{dfn}
Note that being $(\delta_0,\emptyset)$-bounded simply means that no colour appears more than
\[
	\delta_0|V_{[k]}|=\delta_0|V_1|\cdots|V_k|
\]
times.

It is well known that $\bdelta$-bounded colourings for sufficiently small choices of $\delta_0,\dots,\delta_{k-1}$
yield large rainbow subhypergraphs. In fact, being $(\delta_{|J|},J)$-bounded
implies that the number of obstructions to a rainbow coloured subhypergraph consisting of two edges~$e$, $e'$ of the same colour with
$e\cap e'\in V_J$ is at most
\begin{equation}\label{eq:bddJ}
	\delta_{|J|}|V_J||V_{[k]\setminus J}|^2\,.
\end{equation}
This was already exploited in the work of Babai~\cite{B85}, of Lefmann and R\"odl~\cite{LR95}, and of
Alon, Jiang, Miller, and Pritikin~\cite{AJMP03} in a similar context.
The next proposition is based on the same observation.
\begin{prop}\label{prop:simplebdd}
	For $k\geq 2$ and $\bdelta =(\delta_0,\dots,\delta_{k-1}) \in (0,1]^{k}$ let $\phi\colon E(K^{(k)}_{V_1,\dots,V_k})\to \NN$
	be a $\bdelta$-bounded colouring of the complete~$k$-partite~$k$-uniform hypergraph $K^{(k)}_{V_1,\dots,V_k}$
	with vertex partition $V_1\dcup\dots\dcup V_k$.
	If for some positive integer $t\leq\frac{1}{2}\min\big\{|V_1|,\dots,|V_k|\big\}$ we have
	\begin{equation}\label{eq:aSimplerain}
		\delta_j \leq \frac{1}{2^{3k-j}\cdot t^{2k-j-1}}
	\end{equation}
	for every $j=0,\dots,k-1$, then $\phi$ yields a rainbow copy of the
	complete~$k$-partite~$k$-uniform subhypergraph $K^{(k)}_{t,\dots,t}$ with vertex classes of size~$t$.
\end{prop}
\begin{proof}
	For every $i\in[k]$ we choose a random subset $W_i\subseteq V_i$ of size $2t$ and these~$k$ choices are carried out
	independently.
	For a subset~$J\subsetneq [k]$ let~$X_{J}$ be the random variable of the number of pairs of edges
	$\{e,e'\}$ present in the induced subhypergraph $H[W_1,\dots,W_k]$ with
	\[
		e\cap e'\in W_J
		\qand
		\phi(e)=\phi(e')\,.
	\]
	Clearly, we have
	\[
		\EE X_{J}
		\overset{\eqref{eq:bddJ}}{\leq}
		\prod_{j\in J}\frac{2t}{|V_j|}\cdot\prod_{j\not\in J}\frac{2t\cdot(2t-1)}{|V_j|\cdot(|V_j|-1)}\cdot \delta_{|J|}|V_J||V_{[k]\setminus J}|^2\\
		\leq
		\delta_{|J|}\cdot (2t)^{2k-|J|}
		\overset{\eqref{eq:aSimplerain}}{\leq}
		\frac{t}{2^k}
	\]
	and, hence, $\EE\big[\sum_{J\subsetneq [k]}X_J\big] \leq t$. Consequently,
	there exists a choice of sets $W_i\subseteq V_i$ which, after removing one vertex
	for every instance counted by~$X_J$ for every $J\subsetneq [k]$,  induces a rainbow copy of $K^{(k)}_{t,\dots,t}$.
\end{proof}
We shall use a variant of Proposition~\ref{prop:simplebdd}, where we move away from complete partite hypergraphs.
Instead we start with a partite hypergraph of density~$d$ and we are interested in large rainbow subhypergraphs of
similar density.
\begin{prop}\label{prop:densitybdd}
	For $k\geq 2$ let $H=(V_1\dcup \dots\dcup V_k,E)$ be a~$k$-partite~$k$-uniform hypergraph of
	density $d=\frac{|E|}{|V_1|\cdots|V_k|}>0$ and for $\bdelta =(\delta_0,\dots,\delta_{k-1}) \in (0,1]^{k}$
	let $\phi\colon E\to \NN$
	be a $\bdelta$-bounded colouring.
	If for some integer~$m$ we have
	\begin{equation}
		\label{eq:rainbow}
		\frac{2^{k+3}}{d}\leq m \leq \min\big\{|V_1|,\dots,|V_k|\big\}
		\qand
		\delta_j
		<
		\frac{1}{2^{k+1}\cdot m^{2k-j}}
	\end{equation}
	for every $j=0,\dots,k-1$, then~$H$ contains a rainbow subhypergraph of density at least~$d/2$
	with vertex classes $U_j\subseteq V_j$ and $|U_j| = m$ for every $j\in[k]$.
\end{prop}
The proof of Proposition~\ref{prop:densitybdd} parallels the proof of Proposition~\ref{prop:simplebdd}. Roughly speaking,
we show that a randomly chosen subhypergraph inherits the density of~$H$. However, for technical reasons we will
refrain from removing
vertices from the randomly chosen vertex sets and instead we require smaller values of $\delta_j$ in the assumption~\eqref{eq:rainbow}
leading to no expected
obstructions for the rainbow subhypergraph.
\begin{proof}
	For every $i\in[k]$ we choose a random subset $W_i\subseteq V_i$ of size~$m$ and these~$k$ choices are carried out
	independently. Let $H_\cW=H[W_1,\dots,W_k]$ be the random subhypergraph induced on the chosen vertex sets.
	Again for every $J\subsetneq [k]$ we consider the random variable~$X_J$
	counting the pairs of edges~$\{e,e'\}$ spanned in~$H_{\cW}$ with
	\[
		e\cap e'\in W_J\qand \phi(e)=\phi(e')\,.
	\]
	Appealing to the $(\delta_{|J|},J)$-boundedness of $\phi$ we have
	\[
		\EE X_J
		\overset{\eqref{eq:bddJ}}{\leq}
		\prod_{j\in J}\frac{m}{|V_j|}\cdot\prod_{j\not\in J}\frac{m\cdot(m-1)}{|V_j|\cdot(|V_j|-1)}\cdot \delta_{|J|}|V_J||V_{[k]\setminus J}|^2\\
		\leq
		\delta_{|J|}\cdot m^{2k-|J|}
		\overset{\eqref{eq:rainbow}}{<}
		\frac{1}{2^{k+1}}
	\]
	and, owing to Markov's inequality, we arrive at
	\begin{equation}\label{eq:densitybdd1}
		\PP\Big(\sum_{J\subsetneq [k]}X_J \geq 1\Big) < \frac{1}{2}\,.
	\end{equation}

	For the edge density of $H_\cW$ we consider the random variable $Y=\sum_{e\in E}\mathds{1}_{e}=|E(H_\cW)|$.
	Clearly,
	\[
		\EE Y
		=
		\prod_{i=1}^{k}\frac{m}{|V_i|}\cdot |E|
		=
		dm^{k}\,.
	\]
	Recall that
	\begin{align*}
		\Var Y
		=
		\sum_{e,f\in E} \Cov (\mathds1_e,\mathds 1_f)
		=
		\EE Y - \sum_{e\in E}\big(\EE\mathds 1_e\big)^2 + \sum_{\substack{e,f\in E\\ e\neq f}}\big(\EE[\mathds1_e\cdot\mathds1_f]-\EE\mathds{1}_e\cdot\EE\mathds{1}_f\big)\,.
	\end{align*}
	Since $\sum_{e\in E}\big(\EE\mathds 1_e\big)^2\geq 0$ and $\Cov(\mathds 1_e,\mathds 1_f)\leq 0$ for disjoint edges~$e$ and~$f$,
	we can bound the variance by
	\begin{align*}
		\Var Y
		\leq
		\EE Y+\sum_{\emptyset\neq J\subsetneq[k]}\sum_{e\in E}\sum_{\substack{f\in E \\ e\cap f\in V_J}}\EE[\mathds{1}_{e}\cdot\mathds{1}_{f}]\,.
	\end{align*}
	Moreover, for fixed edges~$e$, $f$ with $e\cap f\in V_J$ we have
	\[
		\EE[\mathds{1}_{e}\cdot\mathds{1}_{f}]
		=
		\prod_{j\in J}\frac{m}{|V_j|}\prod_{j\not\in J}\frac{m(m-1)}{|V_j|(|V_j|-1)}
		=
		\PP\big(e\in E(H_\cW)\big)\cdot \prod_{j\not\in J}\frac{m-1}{|V_j|-1}
	\]
	leading to
	\begin{align*}
		\Var Y
		<
		\EE Y+\sum_{\emptyset\neq J\subsetneq[k]}\EE Y\cdot m^{k-|J|}
		< 2^km^{k-1}\EE Y\,.
	\end{align*}
	Consequently, Chebyshev's inequality tells us
	\begin{equation}\label{eq:densitybdd2}
		\PP\Big(Y<\frac{1}{2}\EE Y\Big)
		\leq
		\frac{\Var Y}{(\frac{1}{2}\EE Y)^2}
		<
		\frac{2^{k+2}m^{k-1}}{dm^k}
		=
		\frac{2^{k+2}}{dm}
		\overset{\eqref{eq:rainbow}}{\leq}
		\frac{1}{2}\,.
	\end{equation}
	Owing to the estimates~\eqref{eq:densitybdd1} and~\eqref{eq:densitybdd2}, there exist subsets $U_i \subseteq V_i$ for every $i \in [k]$ each
	of size~$m$ such that no two edges of $H[U_1,\dots,U_k]$ have the same colour under $\phi$ and we have $e(H[U_1,\dots,U_k])\geq dm^k/2$.
\end{proof}

\section{Proof of the main result}
In this section we deduce Theorem~\ref{thm:main} from Propositions~\ref{prop:unbE64}\,--\,\ref{prop:densitybdd}.
\begin{proof}[Proof of Theorem~\ref{thm:main}]
	We first address the general case giving a slightly weaker bound for $k=2$ and $3$. The argument proceeds in three steps: we first fix a hierarchy of auxiliary constants, then distinguish cases according to the minimal index $\js$ at which the colouring fails to satisfy a suitable boundedness condition, and finally we record the sharper choices of constants yielding the improved bounds for $k=2$ and $k=3$. For a fixed integer~$k$, let~$t$
	be sufficiently large and set $c=1/2^{2k+1}$.
	We fix
	auxiliary constants
	\[
		\delta_j
		=
		\Big(\frac{c}{t^{k}}\Big)^{(2kt^k)^{k-1-j}}
		\qqand
		m_j
		=
		\Big(\frac{1}{\delta_j}\Big)^{t^k}
		=
		\Big(\frac{t^{k}}{c}\Big)^{(2k)^{k-1-j}t^{k(k-j)}}
	\]
	for $j=2,\dots,k-1$ and we set
	\[
		\delta_1=\delta_0=\Big(\frac{c}{t^{k}}\Big)^{(2kt^k)^{k-2}}
	\]
	and $\bdelta=(\delta_0,\dots,\delta_{k-1})$.
	Moreover, by assumption of Theorem~\ref{thm:main} we have
	\[
		n=t^{t^{k^2}}\,.
	\]

	The choices above yield the following order of the involved constants
	\[
		2\leq k < \frac{1}{c} < t
		<
		\frac{1}{\delta_{k-1}} < m_{k-1}
		<
		\dots
		<
		\frac{1}{\delta_{2}} < m_{2}
		<
		\frac{1}{\delta_{1}}
		=
		\frac{1}{\delta_{0}}
		<
		n\,.
	\]
	Besides this monotonicity we shall employ a few more relations between these constants. In fact, the proof relies on
	Propositions~\ref{prop:unbE64}\,--\,\ref{prop:densitybdd} and their applications will be justified by the tailored
	inequalities~\eqref{eq:0bdd}--\eqref{eq:jbdd-rainbow} below. These estimates are mainly based on the facts that for $j\geq 2$ we have
	$m_j=\delta_j^{-t^k}$, that
	$\delta_{j-1}=\delta_j^{2kt^k}$, and that we assume that~$t$ is sufficiently large as a function of~$k$.

	We shall apply Proposition~\ref{prop:unbE64} in three different ways and for those applications we rely on
	the following three sets of inequalities
	\begin{equation}\label{eq:0bdd}
		\Big(\frac{\delta_{0}}{4^{k-1}}\Big)^{t^{k-1}}n\geq 2t\,,
	\end{equation}
	\begin{equation}\label{eq:1bdd}
		\frac{\delta_1}{4^{k-1}}\cdot \sqrt{\delta_1 n}
		\geq
		2t
		\qqand
		\Big(\frac{\delta_1}{4^{k-1}}\Big)^{t^{k-1}}n
		\geq
		2t\,,
	\end{equation}
	and for $j=2,\dots,k-1$
	\begin{equation}\label{eq:jbddE64}
		\Big(\frac{\delta^2_j/2}{4^{k-1}}\Big)^{t^{j-1}}m_j
		\geq
		2t
		\qqand
		\Big(\frac{\delta^2_j/2}{4^{k-1}}\Big)^{t^{k-1}}n
		\geq
		2t\,.
	\end{equation}
	Similarly, preparing for an application of Proposition~\ref{prop:simplebdd}, we observe
	\begin{equation}\label{eq:bdd}
		\delta_{k-1}=\frac{1}{2^{2k+1}\cdot t^k}
		\qand
		\delta_j\leq \delta_{k-2} \leq\frac{1}{2^{3k-j}\cdot t^{2k-j-1}}
		\quad\text{for}\ j=0,\dots,k-2\,.
	\end{equation}
	Finally, for an intended application of Proposition~\ref{prop:densitybdd},
	we note that for all $0\leq j< \js\leq k-1$ and $\js\geq 2$ we have
	\begin{equation}\label{eq:jbdd-rainbow}
		\delta_{\js}m_{\js}
		>
			2^{\js+3}
		\qqand
		\frac{\delta_j}{\delta_{\js}}<\frac{1}{2^{\js+1}m_{\js}^{2\js-j}}\,.
	\end{equation}
	This concludes the discussion of the constants.

	We shall show that every colouring $\phi\colon E(K^{(k)}_{n,\dots,n})\to\NN$ yields a canonical copy of~$K^{(k)}_{t,\dots,t}$.
	Let $V_1\dcup\dots\dcup V_k$ be the vertex partition of $K^{(k)}_{n,\dots,n}$.
	Given a colouring $\phi$ we consider several cases depending on the `boundedness properties' of $\phi$.

	Note that, if $\phi$ is indeed $\bdelta$-bounded, then in view of~\eqref{eq:bdd}
	our choice of $\bdelta$ yields a rainbow coloured
	copy of $K^{(k)}_{t,\dots,t}$ by Proposition~\ref{prop:simplebdd}. Consequently, we may assume that there exists a minimal
	index $\js\in\{0,\dots,k-1\}$ such that $\phi$ is not $(\delta_{\js},\js)$-bounded and let $\Js\in[k]^{(\js)}$
	be a set that witnesses this and without loss of generality we may assume $\Js=\{1,\dots,\js\}$.

	If $\js=0$, then one of the colours appears at least $\delta_0n^k$ times. In view of inequality~\eqref{eq:0bdd},
	Proposition~\ref{prop:unbE64} applied with
	\[
		d=\delta_0
		\qqand
		t_1=\dots=t_k=t
	\]
	to the hypergraph $K^{(k)}_{n,\dots,n}$ yields a monochromatic copy of $K^{(k)}_{t,\dots,t}$ in this case.

	If $\js=1$, then there is a set $U\subseteq V_1$
	of size at least $\delta_1 n$ such that every $u\in U$ is contained in at least $\delta_1 n^{k-1}$ edges of the same colour
	and we denote this colour by
	$\l(u)$.

	The pigeonhole principle yields a subset $U_{\star}\subseteq U$ of size at least
	$\sqrt{\delta_1 n}$ such that either all colours~$\l(u)$ for $u\in U_{\star}$ are equal or they are all distinct.
	We consider the~$k$-partite~$k$-uniform hypergraph~$H_{\star}$
	with vertex partition
	\[
		U_{\star}\dcup V_2\dcup\dots\dcup V_k
	\]
	and
	\[
		E(H_{\star})=\bigcup_{u\in U_{\star}}\big\{ e\in V_{[k]}\colon u\in e\tand \phi(e)=\l(u)\big\}\,.
	\]
	Since every vertex $u\in U_{\star}$ has degree at least $\delta_1 n^{k-1}$, the hypergraph~$H_{\star}$ has density at least~$\delta_1$.
	Again we apply Proposition~\ref{prop:unbE64} to the~$k$-partite~$k$-uniform~$H_{\star}$
	having vertex classes of sizes $n_1=|U_{\star}|\geq \sqrt{\delta_1 n}$ and
	$n_2=\dots=n_k=n$, this time with parameters
	\[
		d=\delta_1
		\qqand
		t_1=\dots=t_k=t\,.
	\]
	We note that the assumption~\eqref{eq:aE64} is justified by~\eqref{eq:1bdd} and, hence,
	we obtain either a monochromatic or
	a $\{1\}$-canonical copy of~$K^{(k)}_{t,\dots,t}$.
	This concludes the proof for the cases when $\js\leq 1$.

	It is left to consider the case $\Js=\{1,\dots,\js\}$ for some $\js=2,\dots,k-1$.
	Let $\UJs \subseteq V_{\Js}$ be a set of size at least $\delta_{\js}n^{\js}$ such
	that every $\js$-tuple $S \in \UJs$ extends to at least $\delta_{\js}n^{k-\js}$ edges of
	the same colour and we denote this colour by $\l(S)$.
	We consider the $\js$-partite $\js$-uniform hypergraph~$G$ with vertex
	partition
	\[
		V_1 \dcup \dots \dcup V_{\js} \qand E(G) = \UJs\,.
	\]
	Moreover, we define a colouring $\phis\colon E(G)\to\NN$ through
	\[
		\phis(S) = \l(S)
	\]
	for all~${S \in E(G)}$. Owing to the minimal choice of $\js$, a moment of thought reveals that
	the colouring $\phis$ is $\big(\delta_j/\delta_{\js}, j\big)$-bounded
	for every~${j=0,\dots,\js-1}$. In other words, the colouring~$\phis$ of the $\js$-partite $\js$-uniform
	hypergraph~$G$ is $\big(\delta_0/\delta_{\js},\dots,\delta_{\js-1}/\delta_{\js}\big)$-bounded.
	In view of the estimates~\eqref{eq:jbdd-rainbow}, we may apply Proposition~\ref{prop:densitybdd} to~$G$ with
	\[
		k=\js\,,\quad
		d=\delta_{\js}\,,\quad
		\bdelta=\big(\delta_0/\delta_{\js},\dots,\delta_{\js-1}/\delta_{\js}\big)\,,\qand
		m=m_{\js}\,.
	\]
	This way we obtain a rainbow $\js$-uniform subhypergraph $\Gs$ of
	density at least~$\delta_{\js}/2$ with vertex
	classes~${U_j \subseteq V_j}$ and~$|U_j|=m_{\js}$ for every $j \in [\js]$.

	We now consider the natural~$k$-uniform extension $\Hs$ of $\Gs$ on the vertex partition
	\[
		U_{1}\dcup \dots \dcup U_{\js}\dcup V_{\js+1}\dcup\dots\dcup V_k
	\]
	with
	\[
		E(\Hs)=\bigcup_{S\in E(\Gs)}\big\{ e\in V_{[k]}\colon S\subseteq e\tand \phi(e)=\l(S)\big\}
	\]
	and note that the colouring $\phi$ restricted to $\Hs$ is $\Js$-canonical.
	Moreover, since every $\js$-tuple $S \in E(\Gs)$ extends to at least $\delta_{\js}n^{k-\js}$ distinct
	$k$-tuples of colour $\l(S)$ under~$\phi$, the~$k$-uniform hypergraph~$\Hs$ has density
	at least $\delta_{\js}^{2}/2$.

	Finally, we apply Proposition~\ref{prop:unbE64}
	to the~$k$-partite~$k$-uniform hypergraph~$\Hs$
	having vertex classes of sizes $n_1=\dots=n_{\js}=m_{\js}$ and $n_{\js+1}=\dots=n_k=n$
	with parameters
	\[
		d = \frac{\delta_{\js}^2}{2}
		\qqand
		t_1=\dots=t_k=t\,.
	\]
	Note that the assumption~\eqref{eq:aE64} is justified by the estimates~\eqref{eq:jbddE64}.
	Consequently,  Proposition~\ref{prop:unbE64} yields a $\Js$-canonical copy of
	$K^{(k)}_{t,\dots,t}$. This concludes the proof of Theorem~\ref{thm:main} for $k\geq 4$
	and it is left to discuss the better bounds on~$n$ for the cases $k=2$ and $k=3$.

	For the case $k=2$ one can check that the same proof works for $n=t^{(3+\eps)t}$ for any fixed $\eps>0$
	and  sufficiently
	large~$t$ with $\delta_1=\delta_0=2^{-6}t^{-3}$.  In fact, for graphs the
	proof is somewhat simpler, since the case~$\js\geq 2$ does not arise.

	Similarly, for $k=3$ and $n=t^{30t^3}$ one can check that the choices
	\[
		\delta_2=\frac{1}{2^7t^3}\,,\qquad m_2=t^{7t}\,,\qqand\delta_1=\delta_0=\frac{1}{2^{10}\cdot t^{29t}}
	\]
	satisfy  inequalities~\eqref{eq:0bdd}--\eqref{eq:jbdd-rainbow} for sufficiently large~$t$ and, consequently, the proof presented
	yields the claimed bound in this case.
\end{proof}

\subsection*{Acknowledgement} We are grateful to the reviewers for their detailed and helpful work.


\begin{bibdiv}
	\begin{biblist}

		\bib{AJMP03}{article}{
			author={Alon, Noga},
			author={Jiang, Tao},
			author={Miller, Zevi},
			author={Pritikin, Dan},
			title={Properly colored subgraphs and rainbow subgraphs in edge-colorings
					with local constraints},
			journal={Random Structures Algorithms},
			volume={23},
			date={2003},
			number={4},
			pages={409--433},
			issn={1042-9832},
			review={\MR{2016871}},
			doi={10.1002/rsa.10102},
		}

		\bib{B85}{article}{
			author={Babai, L\'{a}szl\'{o}},
			title={An anti-Ramsey theorem},
			journal={Graphs Combin.},
			volume={1},
			date={1985},
			number={1},
			pages={23--28},
			issn={0911-0119},
			review={\MR{796179}},
			doi={10.1007/BF02582925},
		}

		\bib{DM}{article}{
			author={Dob\'ak, D.},
			author={Mulrenin, E.},
			title={Sharp exponents for bipartite Erd\H{o}s-Rado numbers},
			journal={J. Graph Theory},
			volume={112},
			date={2026},
			pages={96--102},
			issn={0364-9024},
			doi={10.1002/jgt.70016},
		}

		\bib{E64}{article}{
			author={Erd\H{o}s, P.},
			title={On extremal problems of graphs and generalized graphs},
			journal={Israel J. Math.},
			volume={2},
			date={1964},
			pages={183--190},
			issn={0021-2172},
			review={\MR{183654}},
			doi={10.1007/BF02759942},
		}

		\bib{EH89}{article}{
			author={Erd\H{o}s, P.},
			author={Hajnal, A.},
			title={Ramsey-type theorems},
			note={Combinatorics and complexity (Chicago, IL, 1987)},
			journal={Discrete Appl. Math.},
			volume={25},
			date={1989},
			number={1-2},
			pages={37--52},
			issn={0166-218X},
			review={\MR{1031262}},
			doi={10.1016/0166-218X(89)90045-0},
		}

		\bib{EHR65}{article}{
			author={Erd\H{o}s, P.},
			author={Hajnal, A.},
			author={Rado, R.},
			title={Partition relations for cardinal numbers},
			journal={Acta Math. Acad. Sci. Hungar.},
			volume={16},
			date={1965},
			pages={93--196},
			issn={0001-5954},
			review={\MR{202613}},
			doi={10.1007/BF01886396},
		}

		\bib{ER50}{article}{
			author={Erd\H{o}s, P.},
			author={Rado, R.},
			title={A combinatorial theorem},
			journal={J. London Math. Soc.},
			volume={25},
			date={1950},
			pages={249--255},
			issn={0024-6107},
			review={\MR{37886}},
			doi={10.1112/jlms/s1-25.4.249},
		}

		\bib{GMSW}{article}{
			author={Gishboliner, Lior},
			author={Milojevi\'{c}, Aleksa},
			author={Sudakov, Benny},
			author={Wigderson, Yuval},
			title={Canonical Ramsey numbers of sparse graphs},
			journal={SIAM J. Discrete Math.},
			volume={39},
			date={2025},
			number={3},
			pages={1491--1519},
			issn={0895-4801},
			review={\MR{4933867}},
			doi={10.1137/24M1714964},
		}

		\bib{KMV17}{article}{
			author={Kostochka, Alexandr},
			author={Mubayi, Dhruv},
			author={Verstra\"ete, Jacques},
			title={Tur\'an problems and shadows II: Trees},
			journal={J. Combin. Theory Ser. B},
			volume={122},
			date={2017},
			pages={457--478},
			issn={0095-8956},
			review={\MR{3575215}},
			doi={10.1016/j.jctb.2016.06.011},
		}

		\bib{KST54}{article}{
			author={K\H{o}v\'{a}ri, T.},
			author={S\'{o}s, V. T.},
			author={Tur\'{a}n, P.},
			title={On a problem of K.\ Zarankiewicz},
			journal={Colloq. Math.},
			volume={3},
			date={1954},
			pages={50--57},
			issn={0010-1354},
			review={\MR{65617}},
			doi={10.4064/cm-3-1-50-57},
		}

		\bib{LR95}{article}{
			author={Lefmann, Hanno},
			author={R\"{o}dl, Vojt\v{e}ch},
			title={On Erd\H{o}s--Rado numbers},
			journal={Combinatorica},
			volume={15},
			date={1995},
			number={1},
			pages={85--104},
			issn={0209-9683},
			review={\MR{1325273}},
			doi={10.1007/BF01294461},
		}

		\bib{R54}{article}{
			author={Rado, R.},
			title={Direct decomposition of partitions},
			journal={J. London Math. Soc.},
			volume={29},
			date={1954},
			pages={71--83},
			issn={0024-6107},
			review={\MR{65616}},
			doi={10.1112/jlms/s1-29.1.71},
		}

		\bib{R30}{article}{
			author={Ramsey, F. P.},
			title={On a problem of formal logic},
			journal={Proc. London Math. Soc. (2)},
			volume={30},
			date={1930},
			number={4},
			pages={264--286},
			issn={0024-6115},
			review={\MR{1576401}},
			doi={10.1112/plms/s2-30.1.264},
		}

		\bib{RRSSS22}{article}{
			author={Reiher, Chr.},
			author={R\"{o}dl, Vojt\v{e}ch},
			author={Sales, Marcelo},
			author={Sames, Kevin},
			author={Schacht, Mathias},
			title={On quantitative aspects of a canonisation theorem for
					edge-orderings},
			journal={J. Lond. Math. Soc. (2)},
			volume={106},
			date={2022},
			number={3},
			pages={2773--2803},
			issn={0024-6107},
			review={\MR{4498567}},
			doi={10.1112/jlms.12648},
		}

		\bib{S96}{article}{
			author={Shelah, Saharon},
			title={Finite canonization},
			journal={Comment. Math. Univ. Carolin.},
			volume={37},
			date={1996},
			number={3},
			pages={445--456},
			issn={0010-2628},
			review={\MR{1426909}},
		}
	\end{biblist}
\end{bibdiv}
\end{document}